\documentclass[a4paper]{siamart220329}

\usepackage{damacros}

\newcommand{\stsets}[1]{\mathbb{#1}}
\newcommand{\R}{\stsets{R}}

\newcommand{\N}{\stsets{N}}

\renewcommand{\P}{\mathbf{P}}
\DeclareMathOperator{\E}{{\bf E}}
\DeclareMathOperator{\one}{{ 1\hspace*{-0.55ex}I}}

\DeclareMathOperator{\var}{Var}

\DeclareMathOperator{\TLLN}{TLLN}

\newcommand*\diff{\mathop{}\!\mathrm{d}}

\renewcommand{\P}{\mathbf{P}}

\newtheorem{assumption}{Assumption}

\DeclareMathOperator*{\argmin}{arg\,min}

\renewcommand{\epsilon}{\varepsilon}

\renewcommand{\phi}{\varphi}

\def\acknowledgementsname{Acknowledgments}
\newenvironment{acks}[1][\acknowledgementsname]{\section*{#1}}{\par}
 
\if@balayout
    \renewenvironment{acks}[1][\acknowledgementsname]%
        {%
            \vskip0.5\baselineskip
            \small
            {\noindent\normalfont\sffamily\bfseries\acknowledgementsname}\par
            \begingroup\parindent 0pt\parskip 0.5\baselineskip
        }%
        {\endgroup}
\fi
\title{Poisson Hypothesis and large-population limit for networks of spiking neurons}

\author{%
  Daniele Avitabile%
  \thanks{%
    Amsterdam Centre for Dynamics and Computation,
    Vrije Universiteit Amsterdam,
    Department of Mathematics,
    Faculteit der Exacte Wetenschappen,
    De Boelelaan 1081a,
    1081 HV Amsterdam, The Netherlands.
  \protect
    MathNeuro Team,
    Inria branch of the University of Montpellier,
    860 rue Saint-Priest
    34095 Montpellier Cedex 5
    France.
  \protect
  (\email{d.avitabile@vu.nl}, \url{www.danieleavitabile.com}, \url{www.amsterdam-dynamics.nl}).
  }
  \and
  Michel Davydov
  \thanks{%
    Division of Applied Mathematics, Brown University, 182 George Street, Providence, RI 02912
    \protect
    (\email{michel\_davydov@brown.edu}, \url{https://sites.google.com/view/mdavydov/home}).
    }
}

\graphicspath{{./Figures/}}

\begin{document}

\maketitle

\begin{abstract}
  We study mean-field descriptions for spatially-extended networks of linear (leaky) and quadratic
  integrate-and-fire neurons with stochastic spiking times. We consider
  large-population limits of continuous-time Galves-L\"ocherbach (GL) networks with
  linear and quadratic intrinsic dynamics. We prove that that the Poisson Hypothesis
  holds for the replica-mean-field limit of these networks, that is, in a suitably-defined limit, neurons are independent with interaction times replaced by independent time-inhomogeneous Poisson
  processes with intensities depending on the mean firing rates, extending known results to networks with quadratic intrinsic dynamics and resets. Proving that the Poisson Hypothesis holds opens up the
  possibility of studying the large-population limit in these networks. We prove this
  limit to be a well-posed neural field model, subject to stochastic resets. 
\end{abstract}

\section{Introduction}\label{sec:introduction} 

The derivation of mean-field limits of networks of interacting neurons is a central
question in mathematical neuroscience
\cite{ermentroutMathematicalFoundationsNeuroscience2010,
  bressloffWavesNeuralMedia2014,
coombesNeurodynamicsAppliedMathematics2023}, 
and the subject of a continued research activity. Studying the
limiting behaviour of networks of Hodgkin--Huxley neurons with biological realism at
the microscale is mathematically and computationally challenging, and the
existing literature focuses on simplified approaches to treat this problem. Typically,
simplifications are made for the microscopic neuronal dynamic and their coupling.

Neuronal ensembles that reproduce collective firing-rate dynamics can be obtained, for
instance, in spiking neural networks, which couple neurons with an idealised
dynamics. Rather than following the action potential of a neuron during a spike,
neurons in these networks spike instantaneously, an event which determines a reset
of the action potential, and the transmission of a current to connected neurons~\cite{%
vanVreeswijk:1994fb,
vanVreeswijk:1996jf,
ermentrout1998neural,
Ermentrout1998c,
Bressloff:2000uj,
Laing:2001fc,
Ermentrout:2002dl,
Osan:2002jq,
compte2003cellular,
Osan:2004ko,
Mirollo:2006ft,
Gerstner:2008be,
montbrioMacroscopicDescriptionNetworks2015,
laingExactNeuralFields2015,
byrneNextgenerationNeuralMass2020,
avitabileBumpAttractorsWaves2023}. 

Another separate class of simplified models describes neuronal dynamics through firing rates,
rather than neuronal spikes. A neuron's repetitive firing is an emergent feature,
generated by the interplay of intrinsic dynamics and synaptic coupling. Rate models
neglect single spike dynamics and prescribe the evolution of the neuronal or
population firing rate. The resulting models are versatile tools to model collective
dynamics of ensemble of spatially localised or distributed neurons, and we refer to
\cite{ermentroutMathematicalFoundationsNeuroscience2010,
  bressloffWavesNeuralMedia2014,
coombesNeurodynamicsAppliedMathematics2023} and references therein for an overview of
models and applications.

The resulting neural mass or neural field models are heuristic descriptions of
coarse-grained neuronal activity, written as ordinary-, delayed- or
integro-differential equations. In parallel to their adoption as modelling tools,
mathematical techniques have been developed to rigorously derive equations of neural
mass of neural field type as limits of networks of coupled firing rate neurons, in
the presence of noise
\cite{
touboulNoiseInducedBehaviorsNeural2012a,
Fournier_Locherbach_2016,
ditlevsen_locherbach_17,
chevallier2019mean,ZAN_2023}.

A common framework in modelling spiking networks is to prescribe the evolution of a neuronal
state variable, and rules
for their spiking events. Informally, a network of $K$ deterministic leaky
integrate-and-fire (LIF) neurons, is written as
\begin{equation}\label{eq:LIF}
  \tau \frac{d \lambda_i}{dt}(t) 
  = f(\lambda_i(t)) + b + \sum_{j \in \NSet_K} S_{ij}(\lambda_j(t)), \qquad i \in \NSet_K 
\end{equation}
where $\lambda_i(t)$ is the dimensionless voltage of the $i$th neuron at time $t$,
$f(\lambda) = -\lambda$ models an Ohmic leakage current, $b$ a constant external
current, and $\tau$ a characteristic time scale of the
process. The terms $S_{ij}$ represent voltage-dependent currents
received from other neurons in the network. In an integrate-and-fire model, equations
of type \cref{eq:LIF} hold for times $t$ between spiking events, at which resetting
rules are invoked. If $\lambda_k(t)$ crosses a threshold value at time $\tau_k^l$,
then: (i) $\tau_k^l$ is marked as a spiking time, (ii) $\lambda_k(t)$ is
instantaneously reset to a fixed value, and (iii) a change is triggered in the
currents $\{ S_{ik} \colon i \in \NSet_K\}$. 

It is important to note that evolution equations of type \cref{eq:LIF} are also in
use when the state variable of the models are firing rates or synaptic variables,
opposed to voltages. In rate models, however, there are no resets because firing
rates describe frequencies of spiking, not spiking events.

This paper deals with the derivation of a mean field ($K \to \infty$ limit) for a
noisy, spatially-distributed network of integrate-and-fire neurons, with linear or
quadratic dynamics, and subject to resets. Spatially-extended versions of
deterministic LIF neurons \cref{eq:LIF} are known to support patterns of cortical
activity, but the rigorous derivation of a mean-field model for these systems is an open
problem (see \cite{avitabileBumpAttractorsWaves2023} and references therein). In the
past decade, mean fields for networks of \textit{quadratic} integrate-and-fire
neurons (QIFs), in which the leaky current in \cref{eq:LIF} is replaced by a quadratic term $f(\lambda)
= \lambda^2$, have been investigated using the Ott-Antonsen formalism
\cite{montbrioMacroscopicDescriptionNetworks2015,laingExactNeuralFields2015,
byrneNextgenerationNeuralMass2020}. In the latter studies the evolution is
deterministic, but each neuron receives a random background current, extracted from a
Cauchy distribution. In the present article, we will study their dynamics subject to
stochastic spiking times, which makes the Ott-Antonsen formalism not immediately
applicable.

Substantial progress has been made on stochastic networks of type \cref{eq:LIF} when
$\lambda_i(t)$ are interpreted as rates (without resets) and the spiking events are
point processes. One model
of particular relevance is the Galves-L\"ocherbach (GL) model,
that can be thought of as a network of leaky integrate-and-fire neurons with
stochastic threshold reset values \cite{Galves_2013}.
For a fixed population size $K\geq 2$, one introduces the point processes
$N_1,\ldots,N_K$, in which $N_i$ encodes the sequence of spiking times of
neuron $i \in \{ 1,\ldots,K\}$. Neurons are said to follow \textit{continuous-time
GL dynamics} with stochastic intensities $\{ \lambda_i \}_{i=1}^K$
if there exist functions $\Phi_i:\R\times\R\rightarrow\R_+$ and $g_i:\R\rightarrow\R$,
such that
\[
  \lambda_i(t)=\Phi_i\Biggl[\,\sum_{j=1}^K w(i,j)\int_{L_t(i)}^t g_i(t-s)N_j(\diff s), t-L_t(i)\Biggr],
  \qquad t \in \RSet, 
  \qquad i \in \{ 1,\ldots,N \},
\]
where $L_t(i)$ is the last spiking time of neuron $i$ before $t$ and $w(i,j)>0$
represents a synaptic weight and provides the rate increment of neuron $i$ when
neuron $j$ spikes. Stochastic intensities can be informally thought of as
instantaneous neuron firing rates. We refer to \Cref{sec:model_def} for a
formal definition.
A common choice for the functions $\Phi_i$ and $g_i$ is given by
\[
\Phi_i(x,s)=x+b_i+(r_i-b_i)e^{-s/\tau_i},
\qquad 
g_i(t-s)=e^{-(t-s)/\tau_i}, 
\qquad 
b_i,r_i,\tau_i>0.
\]
This system is referred to in literature as the \textit{linear} GL model
\cite{Baccelli_2019}, owing to the linearity of $\Phi_i$ in the variable $x$.

The GL model has been primarily studied in the infinite population
limit \cite{Galves_2013}, notably to construct perfect simulation algorithms
\cite{phi_löcherbach_reynaud-bouret_2023}. However, this study relies on a classical
mean-field approach wherein the synaptic weights $w(i,j)$ are taken to be identical
and equal to $1/K.$ While allowing for a tractable description of the large
population limit in that setting, the mean field approach suffers from the
significant drawback of neglecting the network geometry to do so. An alternative
approach to obtain tractability without considering the large-population limit
has been proposed in \cite{Baccelli_2019} by Baccelli
and Taillefumier, who studied a GL network under the so-called
\textit{Poisson Hypothesis}.
Under the Poisson Hypothesis, the neuron spiking rates $\lambda_i(t)$ are independent processes, and the interaction times between neurons are replaced by independent time-inhomogeneous Poisson processes with intensities depending on the mean neuronal firing rates. Note that this does not imply that the spiking times are Poisson processes. We refer to \cref{def:GL-PH} for a formal definition of these concepts for GL networks. In \cite{Baccelli_2019}, the authors have shown that adopting the Poisson Hypothesis for
linear GL networks in the stationary regime allows for closed forms of the mean
spiking rates.


An important aspect of the Poisson Hypothesis is that it is rarely proven formally, but
usually only conjectured heuristically or checked numerically. For the linear
GL model, it has been proved in \cite{Davydov_2024} that the Poisson
Hypothesis arises at the infinite-replica limit of an interacting particle system
known as the replica-mean-field. This system consists of $M$ copies, or replicas, of
a network of $K$ neurons following linear GL dynamics, in which
interactions between neurons are randomly routed in between replicas. Namely, when
neuron $j$ in replica $n$ spikes, for each $i\in\{1,\ldots,K\}$ such that $w(i,j)>0,$
a replica index $m$ is sampled uniformly and independently from the dynamics, and the
spiking rate of neuron $i$ in replica $m$ is then incremented by $w(i,j).$ This
random routing is at the heart of the replica-mean-field approach and gives rise to
the Poisson Hypothesis in the infinite-replica limit.
\begin{figure}
\label{fig:limit_relations}
\centering
\tikzset{every picture/.style={line width=0.75pt}} 

\begin{tikzpicture}[x=0.75pt,y=0.75pt,yscale=-1,xscale=1]

\draw    (331,73.75) -- (331,113.75) ;
\draw [shift={(331,115.75)}, rotate = 270] [color={rgb, 255:red, 0; green, 0; blue, 0 }  ][line width=0.75]    (10.93,-3.29) .. controls (6.95,-1.4) and (3.31,-0.3) .. (0,0) .. controls (3.31,0.3) and (6.95,1.4) .. (10.93,3.29)   ;
\draw    (330,151.75) -- (330.96,199.75) ;
\draw [shift={(331,201.75)}, rotate = 268.85] [color={rgb, 255:red, 0; green, 0; blue, 0 }  ][line width=0.75]    (10.93,-3.29) .. controls (6.95,-1.4) and (3.31,-0.3) .. (0,0) .. controls (3.31,0.3) and (6.95,1.4) .. (10.93,3.29)   ;

\draw (343,84.4) node [anchor=north west][inner sep=0.75pt]    {$M\rightarrow \infty $};
\draw (255,13) node [anchor=north west][inner sep=0.75pt]   [align=left] {GL model ($K$ neurons)};
\draw (189,48) node [anchor=north west][inner sep=0.75pt]   [align=left] {GL
  replica-mean-field model ($K \displaystyle \times  M$ neurons)};
\draw (166,126) node [anchor=north west][inner sep=0.75pt]   [align=left] {GL model under the Poisson Hypothesis ($K$ neurons)};
\draw (206,211) node [anchor=north west][inner sep=0.75pt]   [align=left] {Neural Field model with stochastic resets};
\draw (347,167.4) node [anchor=north west][inner sep=0.75pt]    {$K\rightarrow \infty $};
\draw (237,84) node [anchor=north west][inner sep=0.75pt]   [align=left] {\textbf{\cref{th:RMF_PoC}}};
\draw (235,168) node [anchor=north west][inner sep=0.75pt]   [align=left] {\textbf{\cref{th:mf_limit}}};

\end{tikzpicture}

\caption{Relations between interacting particle systems}
\end{figure}
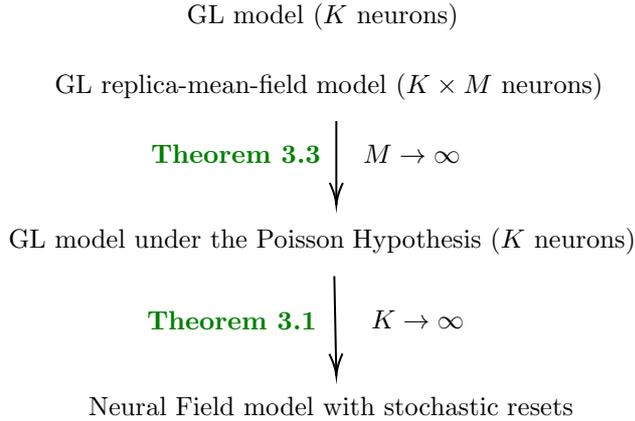

In the present paper we rigorously prove the existence of a mean-field limit for
a network with the following characteristics:
\begin{enumerate}
  \item the intrinsic dynamic of the neuron is of leaky or quadratic type;
  \item the firing times $\tau_k^l$ are given by point processes;
  \item the network is spatially-extended;
  \item neurons in the network are subject to resets.
\end{enumerate}
The derived mean-field resembles a neural field model. To the best of our knowledge,
results of this type are available in the literature only for systems without resets,
and with a linear (leaky) dynamic \cite{chevallier2019mean}.

Our approach differs from classical mean-field results, in that we proceed as
follows: (i) we consider the replica-mean-field version of network dynamics of
interest, (ii) we prove that the Poisson Hypothesis holds in the infinite-replica
limit and, (iii) we proceed with taking the large population limit. 

Because the Poisson Hypothesis holds, neurons in the network are independent, and
this completely circumvents the main technical difficulty in classical mean-field
approaches, namely the proof of asymptotic independence (sometimes referred to as
propagation of chaos) in the large population limit. The main technical contribution
of the paper is a proof that the Poisson Hypothesis holds for networks with quadratic
integrate-and-fire dynamics, which extends the class of models for which the Poisson Hypothesis
holds \cite{davydov2023rmfcFIAP}. 

A schematic of the various interacting particle systems introduced in the paper,
together with the main results is given in \cref{fig:limit_relations}. We introduce
the models of interest in \Cref{sec:model_def}, state our main results in
\Cref{sec:main_result}, and give proofs in \Cref{sec:proof}.

\section{Formal model definition}
\label{sec:model_def}
The goal of this section is to introduce a neural-field-type equation incorporating
stochastic resets with either linear or quadratic autonomous dynamics and to prove
that solutions to this equation can be approximated by an interacting particle system
that we introduce later. This interacting particle system consists of particles evolving
according to Galves-L\"ocherbach dynamics under the Poisson Hypothesis.

To introduce this dynamics, for the finite-dimensional interacting particle system with $K\geq 2$ particles, we introduce the network history  $(\mathcal{F}^K_t)_{t\in \mathbb{R}}$ as an increasing collection of $\sigma$-fields such that
\begin{equation*}
\mathcal{F}_t^\mathbf{N}=\{\sigma(N_1(B_1),...,N_K(B_K))|B_i\in \mathbf{B}(\mathbb{R}), B_i \subset (-\infty,t]\}\subset \mathcal{F}^K_t,
\end{equation*}
where $\mathcal{F}_t^\mathbf{N}$ is the internal history of the $K-$dimensional vector point process $\mathbf{N}=(N_1,\ldots,N_K)$. For the infinite-dimensional system, the network history $(\mathcal{F}_t)_{t\in \mathbb{R}}$ is defined similarly as a collection of $\sigma$-fields such that

\begin{equation*}
\biggl\{\sigma(\bigcup_{k \in \N} N_k(B_k)|B_i\in \mathbf{B}(\mathbb{R}), B_i \subset (-\infty,t]\biggr\}\subset \mathcal{F}^K_t.
\end{equation*}

Recall that the $\mathcal{F}_t$-stochastic intensity $\{\lambda(t)\}_{t \in \mathbb{R}}$ of an associated point process $N$ is the $\mathcal{F}_t$-predictable random process satisfying for all $s<t \in \mathbb{R}:$
\begin{equation*}
\E\left[N(s,t]|\mathcal{F}_s\right]=\E\left[\int_s^t \lambda(u)\diff u\big|\mathcal{F}_s\right],
\end{equation*}
where $\mathcal{F}_t$ is the network history.

Given an interval $\mathcal{S} \subset \R$  and $\mathcal{M}$ a metric space, let $\mathcal{D}(\mathcal{S}:\mathcal{M})$ be the space of càdlàg functions equipped with the Skorokhod topology. Throughout, we write $\mathcal{D}:=\mathcal{D}([0,\infty):\R])$.

We now start by defining the spatially-extended mean-field limit model, the well-posedness of which is the main result of this work. We call this model a neural-field model with stochastic resets.

\begin{definition}[Neural-field model with stochastic resets]
\label{def:NFE_w_resets}
Let $D$ be a bounded subset of $\R$ and let $\tau, T>0.$  Let $w:D\times D \rightarrow \R$ be a continuous function representing the interaction weights, and let $r:D\rightarrow \R^+$ be a measurable function representing the reset values of the neurons. Let $\Phi:\R_+\rightarrow \R$ be a measurable function representing the autonomous evolution of the dynamics in the absence of interactions or resets.
Let $(N_{x})_{x \in D}$ be point processes admitting $t\rightarrow \lambda(x,t)$ as stochastic intensities, where $ \lambda(x,t)$ solves
\begin{equation}
\label{eq:NFE_with:resets}
\begin{split}
\lambda(x,t)=&\lambda(x,0)+\frac{1}{\tau}\int_0^t\Phi(\lambda(x,s))\diff s+\int_D\int_0^t w(x,y)\E[\lambda(y,s)]\diff s\diff y\\
&+\int_0^t\left(r(x)-\lambda(x,s)\right) N_{x}(\diff s), 
\end{split}
\end{equation}    
with $\lambda(x,0)$ being r.v.s satisfying $\E[\lambda(x,0)]<\infty$ for all $x\in D.$
\end{definition}

In this work, we consider two separate regimes for the autonomous evolution dynamics, encapsulated in the following assumption:

\begin{assumption}
\label{Ass:A}
The autonomous evolution function $\Phi:\R_+\rightarrow\R$ from \cref{def:NFE_w_resets} defines either:
\begin{enumerate}
    \item leaky autonomous evolution: fix $b>0$, and let $\Phi(x)=b-x, x\in \R_+$;
    \item quadratic autonomous evolution: fix $b>0$, and let $\Phi(x)=b+x^2, x\in \R_+$.
\end{enumerate}
\end{assumption}

Our goal is to show that there exists a well-posed interacting particle system (IPS) consisting of $K$ particles, or neurons, evolving according to either Galves-L\"ocherbach or quadratic dynamics converging when $K\rightarrow\infty$ to the neural-field dynamics with resets defined by \eqref{eq:NFE_with:resets}. For that purpose, we will have to consider it in the \textit{Poisson Hypothesis} regime, namely, we suppose that the neurons are stochastically independent and hypothetical interaction times between neurons are replaced by independent inhomogeneous Poisson processes, with intensities given by the mean spiking rates of the neurons. 

\begin{definition}[GL dynamics under the Poisson Hypothesis]
\label{def:GL-PH}
Let $D_K$ be a subset of $D$ such that $|D_K|=K.$ Let $(\hat{N}^K_x)_{x \in D_K}$ be a collection of independent Poisson point processes with respective intensities $t\rightarrow \E[\lambda^K(x,t)],$ where $(\lambda^K(x,t))_{x \in D_K}$ are the stochastic intensities of point processes $N^K_{x}$ solving
\begin{equation}
\label{eq:spatial_GL_PH}
\begin{split}
\lambda^K(x,t)=&\lambda^K(x,0)+\frac{1}{\tau}\int_0^t\Phi(\lambda^K(x,s))\diff s+\frac{1}{K-1}\sum_{y \in D_K\setminus\{x\}}w(x,y)\int_0^t \hat{N}^K_{x}(\diff s)\\
&+\int_0^t\left(r(x)-\lambda^K(x,s)\right) N^K_{x}(\diff s),
\end{split}
\end{equation}
where $\Phi$ satisfies \cref{Ass:A} and $\E[\lambda^K(x,0)]<\infty$ for all $x\in D$.
\end{definition}

The Poisson Hypothesis regime is often conjectured, or validated only numerically. For the leaky autonomous evolution, the Poisson Hypothesis is known to be valid for the replica-mean-field limit of GL dynamics, see \cite{Davydov_2024}. In this work, we prove that it is as well the case for the quadratic autonomous evolution. This is the first proof of the Poisson Hypothesis for network dynamics outside the continuous-time fragmentation-interaction-aggregation framework introduced in \cite{davydov2023rmfcFIAP}.

\section{Main results}
\label{sec:main_result}
In this section, we state the two main results of this work: the derivation of the mean-field limit of GL dynamics under the Poisson Hypothesis in both the leaky and quadratic autonomous evolution settings, and the validity of the Poisson Hypothesis for the replica-mean-field limit of GL dynamics with quadratic autonomous evolution, and the validity of the Poisson Hypothesis for the replica-mean-field limit of GL dynamics with quadratic autonomous evolution.

The first result concerns the large-population limit of the dynamics under the Poisson HypothesisWe want to show that when $K \rightarrow \infty,$ these dynamics converge to neural-field-equation-type dynamics defined in \eqref{eq:NFE_with:resets}, with the averaging factor $\frac{1}{K-1}$ introduced for that purpose. 

\begin{theorem}
\label{th:mf_limit}
Let $(x,t)\in D\times [0,T].$ Let $(D_l(x))_{l\geq 1}$ be a sequence of nested subsets of $D$ such that $|D_l(x)|=l$ that densely fill $D,$ that is, for any $l, D_l(x) \subset D_{l+1}(x)$ and $\bigcup_{l=1}^{\infty} D_l(x)=D$ and such that $x \in D_1(x).$
When $K \rightarrow \infty,\lambda^K_x(t)$ defined by $\eqref{eq:spatial_GL_PH}$  converges in probability to $\lambda(x,t)$ defined by $\eqref{eq:NFE_with:resets}.$  
\end{theorem}

Compared with classical mean-field frameworks, the Poisson Hypothesis allows for a much simpler analysis, as the $K$ particles are independent. Thus, the proof of propagation of chaos, a classical component of such results, boils down to a law of large numbers for triangular arrays, detailed in \Cref{sec:proof_mf_limit}

Thus, the main technical difficulty of the current work lies in establishing the Poisson Hypothesis. By that, we mean introducing an interacting particle system such that, when properly scaled, it converges to the Poisson Hypothesis dynamics \eqref{eq:spatial_GL_PH}. When the autonomous evolution function from \cref{Ass:A} is leaky, this is already known: in \cite{Davydov_2024}, it is shown that the Poisson Hypothesis arises at the replica-mean-field limit of GL dynamics. The open question solved by the next theorem is the generalization of that result to the case where the autonomous evolution function from \cref{Ass:A} is quadratic.

We start by rigorously introducing replica-mean-field dynamics of GL neurons.

\begin{definition}[Replica-mean-field dynamics of GL neurons]
\label{def:RMF_dynamics}    
Let $M\geq 2.$ Let $D_K$ be a subset of $D$ such that $|D_K|=K.$ The $M-$replica quadratic Galves-L\"ocherbach dynamics on $D_K$ are a collection of $MK$ point processes $(N^{M,K}_{m,x})_{1\leq m \leq M,1\leq i\leq K}$ admitting $\mathcal{F}^K_t-$stochastic intensities $\lambda_{m}(x,t)$ satisfying the following system of equations:
\begin{equation}
\label{eq:spatial_GL_RMF}
\begin{split}
\lambda^{M,K}_m(x,t)=&\lambda^{M,K}_m(x,0)+\frac{1}{\tau}\int_0^t\Phi(\lambda^{M,K}_m(x,s))\diff s\\
&+\frac{1}{K-1}\sum_{y \in D_K\setminus\{x\}}\sum_{n\neq m}w(x,y)\int_0^t \one_{\{V^M_n(y,x,s)=m\}} N^{M,K}_{n,y}(\diff s)\\
&+\int_0^t\left(r(x)-\lambda^{M,K}(x,s)\right) N^{M,K}_{m,x}(\diff s), 
\end{split}
\end{equation}
where $\{V^M_n(y,x,t)\}_{t\in \mathbb{R}}$ are $(\mathcal{F}_t)$-predictable processes for $1 \leq m \leq M, x,y \in D_K,$ such that for each  each point $T^{K,M}_{n,y}$ of the process $N^{M,K}_{n,y}$, the real-valued random variables $\{V^M_n(y,x,t,T^{K,M}_{n,y})\}_y$ are mutually independent, independent from the past, i.e. from $\mathcal{F}_s$ for $s<T^{K,M}_{n,y},$ and uniformly distributed on $\{1,...,M\}\setminus\{n\},$ and
$\lambda^{M,K}_m(x,0)$ are $M-$exchangeable random variables with finite first moment.
\end{definition}

This allows us to state our second main theorem:
\begin{theorem}[Propagation of chaos for RMF dynamics]
\label{th:RMF_PoC}
For $t\geq 0$, there exists $C>0$ such that for all $m \in \{1,\ldots,M\}$ and all $x\in D_K$,
\begin{equation*}
    d_{TV}(\lambda^{M,K}_m(x,t),\lambda^K(x,t))\leq \frac{C}{\sqrt{M}},
\end{equation*}
where $\lambda^{M,K}_m(x,t)$ is defined by \eqref{eq:spatial_GL_RMF} and $\lambda^K(x,t)$ is defined by \eqref{eq:spatial_GL_PH}.
\end{theorem}

The proof of this result is detailed in \Cref{sec:proof_RMF_POC}.

\section{Proofs}
\label{sec:proof}

\subsection{Proof of \cref{th:mf_limit}}
\label{sec:proof_mf_limit}

Using notation from \eqref{eq:spatial_GL_PH}, for $y \in D_K\setminus \{x\},$ let
\begin{equation}
\label{eq:def:summand}
A^K(x,y,t):=\frac{w(x,y)}{K-1}\int_0^t \hat{N}^K_{y}(\diff s),
\end{equation}
and let

\begin{equation}
\label{eq:emp_mean}
A^K(x,t):=\sum_{y \in D_K(x)\setminus \{x\}}A^K(x,y,t)=\sum_{y \in D_K(x)\setminus \{x\}}\frac{w(x,y)}{K-1}\int_0^t \hat{N}^K_{y}(\diff s).
\end{equation}

From the mapping theorem, it is clear that to prove the theorem, it is sufficient to prove the convergence of $A^K(x,t)$ when $K\rightarrow \infty$, which is equivalent to proving a law of large numbers.

Since the processes $(\hat{N}_x)_{x \in D}$ are independent, we will apply the following law of large numbers for arrays of rowwise independent random variables, reformulated from \cite{HuMoriczTaylor89}[Theorem 2];
\begin{lemma}
\label{lem_lln_array}
For $n\in \N,$ let $(X_{n,j})_{1\leq j \leq n}$ be a triangular array of rowwise independent random variables. Let $S_n=\sum_{j=1}^n (X_{n,j}-\E[X_{n,j}]).$
Suppose that
\begin{enumerate}
    \item There exists $a\in \R$ such that, when $n\rightarrow \infty,$
    \begin{equation}
    \label{eq:slln_cond1}
    \sum_{j=1}^n \E[X_{n,j}]\rightarrow a;
    \end{equation}
    \item When $n\rightarrow \infty,$
    \begin{equation}
    \label{eq:slln_cond2}
    \var(S_n)\rightarrow 0.
    \end{equation}
\end{enumerate}
Then $S_n\xrightarrow{\P} 0$ when $n\rightarrow \infty,$ where $\xrightarrow{P}$ denotes convergence in probability.
\end{lemma}
We now proceed to prove \cref{th:mf_limit}.
\begin{proof}
Fix $k\in \N$ and let
\begin{equation*}
\mu_K(x):=\sum_{y \in D_K(x)\setminus \{x\}} \delta_y.
\end{equation*}
Then, we have
\begin{equation*}
\E[A^K(x,t)]=\frac{1}{K-1}\int_0^t\int_D w(x,y)\E[\lambda(y,s)] \mu_K(x)(\diff y)\diff s.
\end{equation*}

Since $D$ is bounded and by definition of $(D_K(x))$, we have $\mu_K(x) \Rightarrow \lambda_D,$ where $\lambda_D$ is the Lebesgue measure on $D$ and $\Rightarrow$ denotes weak convergence.

Since $w$ is bounded and continuous, from \cite{billingsley1968}[Theorem 25.8], we have
\begin{equation*}
\E[A^K(x,t)] \rightarrow \int_0^t\int_D w(x,y)\E[\lambda(y,s)]\diff s \diff y
\end{equation*}
when $K\rightarrow \infty,$
and therefore, condition 1 of \cref{lem_lln_array} is satisfied.

Now, note that
\begin{equation*}
\E[(A^K(x,t))^2]=\frac{1}{(K-1)^2}\int_0^t\int_D w^2(x,y)\E[\lambda(y,s)] \mu_K(x)(\diff y)\diff s.
\end{equation*}
Since $w$ is a continuous function on $D$ and $D$ is bounded, $\E[(A^K(x,t))^2]\rightarrow 0$ when $K \rightarrow \infty.$ Condition 2 from \cref{lem_lln_array} then follows, as does in turn the theorem.
\end{proof}

\subsection{Proof of \cref{th:RMF_PoC}}
\label{sec:proof_RMF_POC}

\cref{th:RMF_PoC} is inspired by an equivalent result for continuous-time fragmentation-interaction-aggregation processes (cFIAP) introduced in \cite{davydov2023rmfcFIAP}. However, when the autonomous evolution function from \cref{Ass:A} is quadratic, the resulting dynamics do not belong to the cFIAP class as the quadratic autonomous evolution is increasing between two fragmentation times. As such, the goal of this section is to show that the main result of \cite{davydov2023rmfcFIAP}, namely the convergence of the $M-$replica-mean-field model to dynamics under the Poisson Hypothesis as $M\rightarrow\infty,$ extends to this quadratic model.

The proof of a result equivalent to \cref{th:RMF_PoC} in \cite{davydov2023rmfcFIAP} relies on the following approach: the total variation distance between the law of the interaction term between two nodes in two distinct replicas and the law of a Poisson r.v. is bounded using the Chen-Stein \cite{Chen75} method. Then, a fixed-point approach is used to show that the bound goes to 0 when the number of replicas goes to infinity.

We first note that even in the quadratic case, the $M-$replica-mean-field dynamics admit finite first moments. We state it in the following lemma and implicitly use it in what follows.
\begin{lemma}
\label{lem:finite_mean}
For all $x\in D_K,t\geq 0$ and $m\in\{1,\ldots,M\},$ when $\Phi$ is the quadratic autonomous evolution function, $\E[\lambda_m^{M,K}(x,t)]<\infty.$ 
\end{lemma}

\begin{proof}
To improve readability, we omit superscripts and write $\lambda_m$ for $\lambda_m^{M,K}$ in this proof.
By exchangeability between replicas, we have for any $x\in D_K$
\begin{equation*}
\begin{split}
\E[\lambda_m(x,t)]=&\E[\lambda_m(x,0)]+ct+\int_0^t\E[\lambda^2_m(x,s)] \diff s\\
&+ \sum_{y\in D_K\setminus\{x\}}w(x,y)\int_0^t \E[\lambda_m(y,s)]\diff s\\
&+r(x)\int_0^t \E[\lambda_m(x,s)]\diff s -\int_0^t \E[\lambda^2_m(x,s)]\diff s.
\end{split}
\end{equation*}
Using the boundedness of $w$ and setting $x(s)=\argmin_{y \in D_K} \E[\lambda_m(x,s)],$
we have the existence of a constant $K>0$ such that
\begin{equation*}
\E[\lambda_m(x(t),t)]\leq \E[\lambda_m(x(t),0)]+ct+(K+r(x(t)))\int_0^t \E[\lambda_m(x(s),s)]\diff s.
\end{equation*}
From Gronwall's lemma, we get that
\begin{equation*}
\E[\lambda_m(x(t),t)]\leq (\E[\lambda_m(x(t),0)]+cT')e^{(K+r(x(t)))T'},
\end{equation*}
which by definition of $x(t)$ and the assumption on the initial conditions in \cref{def:RMF_dynamics} gives that for any $t\in [0,T'],$
\begin{equation*}
\E[\lambda_m(x,t)]<\infty.
\end{equation*}

\end{proof}

Since the interaction term in the $M$-replica dynamics which corresponds to the third term in \eqref{eq:spatial_GL_RMF} is identical to the leaky case and \cref{lem:finite_mean} holds, there is no issue adapting the Chen-Stein bound from the original proof to the quadratic case. Namely, we have the following result:

\begin{lemma}
\label{lem_poisson_chenstein_bound_cFIAP}
Let $M>1$. Let $m \in \{1,\ldots,M\}$ and $x \in D_K$. For $y \in D_K\setminus\{x\}$ and $t\in \R^+$, let $A(y,x,t)=\sum_{n \neq m}\sum_{k\leq N_{n,y}^M([0,t))}w(x,y)\one_{V_n^M(y,x,t)=m}$ with $w:D_K\times D_K\rightarrow \R$ continuous, and let $\hat{A}(y,x,t)$ be independent Poisson random variables with means $\E[N_{1,y}^M([0,t))]$. Then, there exists a positive finite constant $K$ such that
\begin{equation}
\label{eq:chen_poisson_bound_cFIAP}
\begin{split}
&d_{TV}(A(y,x,t),\hat{A}(y,x,t)) \leq \\
&K\Bigg(\bigg(1\wedge\frac{0.74}{\sqrt{\E[N_{1,y}^M([0,t))]}}\bigg)\frac{1}{M-1}\E\left[\left|\sum_{n \neq m}\left(\E[N_{n,y}^M([0,t))]-N_{n,y}^M([0,t))\right)\right|\right]\\
&+\frac{1}{M-1}\left(1\wedge \frac{1}{\E[N_{1,y}^M([0,t))]}\right)\E[N_{1,y}^M([0,t))]\Bigg).
\end{split}
\end{equation}
\end{lemma}

Similarly to \cite{davydov2023rmfcFIAP}, in order to prove \cref{th:RMF_PoC}, it is then sufficient to show that the random variables $(N_{1,y}^M([0,t)))_{y\in D_K}$ satisfy the following law of large numbers:
\begin{definition}
\label{def:TLLN}
Let $M \in \N$. Let $(X^M_n)_{1 \leq n\leq M}$ be $M$-exchangeable random variables with finite expectation. We say they satisfy an $\mathcal{L}^1$ triangular law of large numbers, which we denote $\TLLN(X^M_n)$, if there exists $C>0$ such that
\begin{equation}
\label{eq:TLLN_cFIAP}
\E\left[\left|\frac{1}{M-1}\sum_{n=1}^M (X^M_n-\E[X^M_n])\right|\right]\leq \frac{C}{\sqrt{M}}
\end{equation}
and there exists a random variable $\tilde{X}$ and a constant $D>0$ such that for all $1\leq n_1 \neq n_2\leq M$ and for all $B_1, B_2\in \mathcal{B}(\R),$
\begin{equation}
\label{eq:TLLN_cFIAP_2}
|\P(X^M_{n_1} \in B_1, X^M_{n_2}\in B_2)-\P(\tilde{X} \in B_1)\P(\tilde{X} \in B_2)|\leq \frac{D}{\sqrt{M}}.
\end{equation}
\end{definition}
In \cite{davydov2023rmfcFIAP}, this is proved using a fixed-point approach. Unfortunately, that approach crucially uses a Cramer condition (finiteness of exponential moments) to bound tail probabilities that a priori has no reason to hold in the quadratic case. Therefore, we will here proceed in two steps. First, we introduce some auxiliary RMF dynamics with a fixed threshold, for which it is easy to show the Cramer condition, and consequently the TLLN. Second, we show that for any $t\geq 0,$ it is possible to construct auxiliary dynamics that coincide with the quadratic dynamics on $[0,t]$, from which in turn we deduce the TLLN property for the PH dynamics with quadratic autonomous evolution \eqref{eq:spatial_GL_RMF}.

We start by defining the threshold $M$-replica quadratic GL dynamics on $D_K$ as a collection of $MK$ point processes $(N^{C,M,K}_{m,x})_{1\leq m \leq M,1\leq i\leq K}$ admitting $\mathcal{F}^K_t-$stochastic intensities $\lambda_{m}(x,t)$ satisfying the following system of equations:
\begin{equation}
\label{eq:spatial_threshold_GL_quad_RMF}
\begin{split}
\lambda^{C,M,K}_m(x,t)=&\lambda^{M,K}_m(x,0)+\frac{1}{\tau}\int_0^t\Phi(\lambda^{C,M,K}_m(x,s))\diff s\\
&+\frac{1}{K-1}\sum_{y \in D_K\setminus\{x\}}\sum_{n\neq m}w(x,y)\int_0^t \one_{\{V^M_n(y,x,s)=m\}} N^{C,M,K}_{n,y}(\diff s)\\
&+\int_0^t\left(r(x)-\lambda^{C,M,K}(x,s^-)\right) N^{C,M,K}_{m,x}(\diff s)\\
&+(r(x)-\lambda^{C,M,K}(x,t^-))\one_{\{\lambda^{C,M,K}(x,t)>C\}}, 
\end{split}
\end{equation}
where $C\in R^+$, $\{V^M_n(y,x,t)\}_{t\in \mathbb{R}}$ are $(\mathcal{F}_t)$-predictable processes for $1 \leq m \leq M, x,y \in D_K$ such that for each point $T^{C,K,M}_{n,y}$ of the process $N^{C,M,K}_{n,y}$, the real-valued random variables $\{V^M_n(y,x,t,T^{C,K,M}_{n,y})\}_y$ are mutually independent, independent from the past, i.e. from $\mathcal{F}_s$ for $s<T^{C,K,M}_{n,y},$ and uniformly distributed on $\{1,...,M\}\setminus\{n\}.$

De facto, the threshold dynamics \eqref{eq:spatial_threshold_GL_quad_RMF} are the same as the quadratic RMF dynamics \eqref{eq:spatial_GL_RMF}, with the addition of resets when the intensity reaches the deterministic threshold $C$. These dynamics are well-defined for the same reason that \eqref{eq:spatial_GL_RMF} are, and it is clear that they will satisfy the Cramer condition. Therefore, simply by following the proof for the linear/leaky case in \cite{davydov2023rmfcFIAP}, we obtain the following result:

\begin{lemma}
\label{lem_TLLN_threshold_dynamics}
Let $N^{C,M,K}_{m,x}$ be point processes defined as the threshold $M$-replica quadratic GL dynamics on $D_K$ introduced in \eqref{eq:spatial_threshold_GL_quad_RMF}. Then, for any $t\geq 0$ and any $x\in D_K,$ $\TLLN(N^{C,M,K}_{1,x}([0,t)))$ holds.
\end{lemma}

Next, we show the following tail bound for the RMF dynamics with quadratic autonomous evolution \eqref{eq:spatial_GL_RMF}. This will be afterwards be used to couple them with the threshold dynamics.

\begin{lemma}
\label{lem_quad_tail_bound}
Let $L>0.$ Then, for any $x\in D_K$ and $t\geq 0, \P(\lambda^{M,K}_m(x,t)>L)\leq \frac{1}{\sqrt{1+L^2}},$ where $\lambda^{M,K}_m(x,t)$ is defined by \eqref{eq:spatial_GL_RMF} when the autonomous evolution is quadratic.
\end{lemma}

Before proving this result, we recall the following Poisson embedding representation for point processes with a stochastic intensity \cite[Section 3]{BremMass96}:
\begin{lemma}
\label{lem_poisson_embedding}
Let $N$ be a point process on $\mathbb{R}$. Let $(\mathcal{F}_t)$ be the internal history of $N$. Suppose $N$ admits a $(\mathcal{F}_t)$-stochastic intensity $\{\mu(t)\}_{t \in \mathbb{R}}.$ Then there exists a Poisson point process $\overline{N}$ with intensity 1 on $\mathbb{R}^2$ such that for all $B \in \mathcal{B}(\mathbb{R}),$
\begin{equation*}
    N(B)=\int_{B \times \mathbb{R}}\one_{[0, \mu(s)]}(u)\overline{N}(\diff s\times \diff u).
\end{equation*}
\end{lemma}

\begin{proof}[Proof of \cref{lem_quad_tail_bound}]
Denote $\tau^{K,M}_{m,x}(t)=\max(\sup_{s\in [0,t)} \{s\in N^{K,M}_{m,x}\},0)$ the last reset time before $t$, that is, the last point of the point process $N^{K,M}_{m,x}$ before $t.$ Note that two points of $N^{K,M}_{m,x}$ are at distance at most $\frac{\pi}{2}$ due to the quadratic autonomous evolution, therefore $\tau^{K,M}_{m,x}(t)$ is always greater than 0 if $t>\frac{\pi}{2}.$
Consider $\overline{N}^{M,K}_{m,x}$ the Poisson embedding of $N^{M,K}_{m,x}$ as in \cref{lem_poisson_embedding}. It is then clear that for $L>0,$ the event ``$\lambda^{M,K}_m(x,t)>L$" coincides with the event ``there are no points of $\overline{N}^{M,K}_{m,x}$ under the curve of $\lambda^{M,K}_m(x,t)$ between $\tau^{K,M}_{m,x}(t)$ and $t$".
This observation leads to the fact that the event ``$\lambda^{M,K}_m(x,t)>L$" is included in the event ``there are no points of $\overline{N}^{M,K}_{m,x}$ under the curve of $s\rightarrow t-\arctan(L)+\tan(s)$ between $t-\arctan(L)$ and $t$".

Therefore,
\begin{equation*}
\P(\lambda^{M,K}_m(x,t)>L)\leq e^{-\int_{0}^{\arctan(L)}\tan(x)\diff x}.
\end{equation*}
An elementary computation gives $\int_{0}^{\arctan(L)}\tan(x)\diff x=-\log(\frac{1}{\sqrt{1+L^2}}).$
This leads to
\begin{equation*}
\P(\lambda^{M,K}_m(x,t)>L)\leq \frac{1}{\sqrt{1+L^2}},
\end{equation*}
which completes the proof.
\end{proof}

\begin{proof}[Proof of \cref{th:RMF_PoC}]
For $L>0,$ let $\tau_L=\inf \{t\geq 0: \lambda_{m,i}(t)\geq L\}.$ By \cref{lem_quad_tail_bound}, $\tau_L\rightarrow \infty$ when $L\rightarrow \infty.$ For a process $X$ and $\tau$ a stopping time, denote $X_\tau(t)=X_{\min(t,\tau)}$ the process stopped at $\tau.$ 
By definition of $\tau_L,$ for every $t\geq 0$ and every $x\in D_K, \lambda^{M,K}_{\tau_L,m}(x,t)=\lambda^{L,M,K}_{\tau_L,m}(x,t).$
Combining these two observations, we have the existence of $C>0$ s.t. for all $t>0,$
\begin{equation}
\label{eq:threshold_dyn_equality}
\P(\lambda^{M,K}_m(x,t)=\lambda^{C,M,K}_m(x,t))=1.
\end{equation}

To show \cref{th:RMF_PoC}, we need to show $\TLLN(N_{m,x}^{M,K}(t))$ for every $t>0.$
Using the triangular inequality,
\begin{equation*}
\begin{split}
\E\big[|\frac{1}{M-1\sum_{n=1}^M(N_{m,x}^{M,K}(t)-\E[N_{m,x}^{M,K}(t)])}|\big]\leq& \frac{1}{M-1}\big(\sum_{n=1}^M\E[N_{m,x}^{M,K}(t)-N_{m,x}^{C,M,K}(t)]\\
    &+\sum_{n=1}^M\E[N_{m,x}^{C,M,K}(t)-\E[N_{m,x}^{C,M,K}(t)]]\\
    &+\sum_{n=1}^M(\E[N_{m,x}^{C,M,K}(t)]-\E[N_{m,x}^{M,K}(t)])\big).
\end{split}
\end{equation*}
with $C$ as in \eqref{eq:threshold_dyn_equality}.
The first term is equal to 0 by \eqref{eq:threshold_dyn_equality}. There exists a constant $K>0$ s.t. the second term is bounded by $\frac{K}{\sqrt{M}}$ by \cref{lem_TLLN_threshold_dynamics}. From the explicit derivation of $K(C)$ detailed in \cite{davydov2023rmfcFIAP} The third term can be made arbitrarily small by \cref{lem_quad_tail_bound}. This completes the proof.
\end{proof}
\begin{acks}
MD would like to thank Kavita Ramanan for helpful discussions and feedback. MD was supported by the Office of Naval Research under the Vannevar Bush Faculty Fellowship N0014-21-1-2887. 
DA acknowledges funding from the Dutch Research Council through the NWO Grant WC.019.009
“Spatio-temporal canards and delayed bifurcations in continuous neurobiological
networks”. DA’s work was also supported by the National Science Foundation under
Grant No. DMS-1929284 while the author was in residence at the Institute for
Computational and Experimental Research in Mathematics in Providence, RI, during the
“Math + Neuroscience: Strengthening the Interplay Between Theory and Mathematics"
programme.
\end{acks}

\bibliographystyle{siamplain}
\bibliography{biblio}
\end{document}